\documentclass{article}[12pt]
\usepackage{amsmath}
\usepackage{amsfonts}
\usepackage{amssymb}
\usepackage{appendix}
\usepackage{amsthm }
\usepackage{footnote}
\usepackage{fullpage}

\usepackage[numbers,colon]{natbib}

\usepackage{enumerate}
\newtheorem{Thm}{Theorem}[section]

\newtheorem{Lem}[Thm]{Lemma}

\makeatletter
\def\blfootnote{\xdef\@thefnmark{}\@footnotetext}
\makeatother

\theoremstyle{definition}
\newtheorem{Def}[Thm]{Definition}

\theoremstyle{remark}
\newtheorem{Rem}[Thm]{\bf{Remark}}

\newcommand{\ConvD}{\overset{d}{\rightarrow}}

\newcommand{\E}{\mathbb{E}}
\newcommand{\mathbd}{\boldsymbol}

\title{The universality of homogeneous polynomial forms\\ and critical limits}

\author{Shuyang Bai \qquad Murad S. Taqqu
}

\begin{document}
\maketitle
\begin{abstract}
 \citet{nourdin:peccati:2010:invariance} established the following universality result: if a sequence of off-diagonal homogeneous polynomial forms in i.i.d.\ standard normal random variables converges in distribution to a normal, then the  convergence also holds if one replaces these i.i.d.\ standard normal random variables in the polynomial forms by any independent standardized random variables with uniformly bounded third absolute moment. The  result, which was stated for polynomial forms with a finite number of terms, can be extended to allow an infinite number of terms in the polynomial forms. Based on a contraction criterion derived from this extended universality result, we prove a central limit theorem for a strongly dependent nonlinear processes, whose memory parameter lies at the boundary between short and long memory.
 \end{abstract}
\blfootnote{
\begin{flushleft}
\end{flushleft}
\textbf{2010 AMS Classification:} 60F05\\
}
\section{Introduction}
In \citet{nourdin:peccati:2010:invariance}, a universality result was established for the following off-diagonal homogeneous polynomial form
\begin{equation}\label{eq:finite form}
Q_k(N_n,f_n,\mathbf{X}):=\sum_{1\le i_1,\ldots,i_k\le N_n} f_n(i_1,\ldots,i_k) X_{i_1}\ldots X_{i_k},
\end{equation}
where $f_n$ is a sequence of symmetric functions on $\mathbb{Z}_+^k$ \emph{vanishing on the diagonals} ($f_n(i_1,\ldots,i_k)=0$ if $i_p=i_q$ for some $p\neq q$), and $\mathbf{X}=(X_1,X_2,\ldots)$ is a sequence of standardized independent  random variables, and $N_n$ is a finite sequence such that $N_n\rightarrow\infty$ as $n\rightarrow\infty$.

The universality result says that if $\mathbf{Z}=(Z_1,Z_2,\ldots)$ is an i.i.d.~standard normal sequence and $Q_k(N_n,f_n,\mathbf{Z})$ converges weakly to a normal distribution as $n\rightarrow\infty$, then the same weak convergence to normal holds if $\mathbf{Z}$ is replaced by $\mathbf{X}$, where $\mathbf{X}$ is any  standardized independent  sequence with some uniform higher moment bound.

It is natural  to try to eliminate the finiteness of $N_n$  in the preceding result. This extension was mentioned in Remark 1.13 of \citet{nourdin:peccati:2010:invariance}, but was not explicitly done. One would encounter a number of difficulties if one were to extend the method of proof used for finite $N_n$ to $N_n=\infty$. We will note, however, that this extension can be easily achieved using a simple approximation argument. We find it valuable to have such an extension and the corresponding contraction criterion (Theorem \ref{Thm:practical}) since it can be directly applied to limit theorems in the context of long memory.

We consider such an application in Section \ref{Sec:boundary} where we suppose that
\[
f_N(i_1,\ldots,i_k)=\frac{1}{A(N)}\sum_{n=1}^Na(n-i_1,\ldots,n-i_k)1_{\{-\infty<i_1<n,\ldots,-\infty<i_k<n\}},
\]
and where the function $a(\cdot)$ behaves essentially like a homogeneous function with exponent $\alpha$. The resulting polynomial form $Q_k(f_N)$ is then the partial sum of a stationary process. The exponent $\alpha$ is chosen in such a way that the corresponding stationary process lives on the boundary between short and long memory. We use the contraction criterion to prove that a central limit theorem holds but with the nonstandard normalization $\sqrt{N\ln N}$. This delicate case seems  difficult to treat  otherwise.

The paper is organized as follows. In Section \ref{Sec:universality}, we state the and prove the extension of the universality result (Theorem \ref{Thm:universality ext}), and as a byproduct, a  criterion for  asymptotic normality (Theorem \ref{Thm:practical}). In Section \ref{Sec:setting}, we state the critical limit theorem obtained by applying the criterion. In Section \ref{Sec:proof boundary k>=2} and \ref{Sec:proof linear} we give the proofs.

\section{Universality of homogeneous polynomial forms}\label{Sec:universality}

Let $\ell^2(\mathbb{Z}^k)$, $k\ge 1$, denote the space of symmetric square summable functions on $\mathbb{Z}^k$ \emph{vanishing} on the diagonals equipped with the discrete $L^2$ norm.
Let $\mathbf{X}=(X_1,X_2,\ldots)$ be a sequence of  independent random variables satisfying $\E X_i=0$ and $\E X_i^2=1$.
By modifying the notation (\ref{eq:finite form}), one defines for $f\in\ell^2(\mathbb{Z}^k) $:
\begin{equation*}
Q_k(f,\mathbf{X}):=\sum_{-\infty< i_1,\ldots,i_k <\infty} f(i_1,\ldots,i_k) X_{i_1}\ldots X_{i_k}.
\end{equation*}
One has
\[
\E Q_k(f,\mathbf{X})=0.
\]
Consider now two homogeneous polynomial forms $Q_{k_1}(f_1,\mathbf{X})$ and $Q_{k_2}(f_2,\mathbf{X})$, where $f_1\in \ell^2(\mathbb{Z}^{k_1})$ and $f_2\in \ell^2(\mathbb{Z}^{k_2})$. Then the covariance of $Q_{k_1}(f_1,\mathbf{X})$ and $Q_{k_2}(f_2,\mathbf{X})$ is
\begin{align}
\langle f_1,f_2 \rangle:=&~~\E Q_k(f_1,\mathbf{X})Q_k(f_2,\mathbf{X})\label{eq:cov homo}\\=
&
\begin{cases}
k!\sum_{-\infty<i_1,\ldots,i_k<\infty} f_1(i_1,\ldots,i_k)f_2(i_1,\ldots,i_k),\quad &\text{if }k_1=k_2=k;\\
0   &\text{if }k_1\neq k_2.
\end{cases}\label{eq:cov fun}
\end{align}

\medskip
We then have the following extension of \citet{nourdin:peccati:2010:invariance} Theorem 1.2:
\begin{Thm}\label{Thm:universality ext}
For each  $j=1,\ldots,m$, suppose that $k_j\ge 2$,  and let $f_{n,j}(\cdot)$ be a sequence of  functions in $\ell^2(\mathbb{Z}^{k_j})$. Let $\Sigma$ be an $m\times m$ symmetric non-negative definite matrix whose each diagonal entry is positive. Assume in addition that
\begin{equation}\label{eq:second moment bound}
 \sup_n \sum_{-\infty< i_1,\ldots,i_{k_j}<\infty} f_{n,j}(i_1,\ldots,i_{k_j})^2<\infty.
\end{equation}
 Then the following two statements are equivalent:
\begin{enumerate}
\item For \emph{every} sequence $\mathbf{X}=(X_{1},X_{2},\ldots)$  where $X_{1},X_{2},\ldots$ are independent random variables satisfying $\E X_{i}=0, \E X_{i}^2=1$, and
\begin{equation}\label{eq:uniform bound third moment}
\sup_{i} \E|X_{i}|^3<\infty,
\end{equation}
the following joint weak convergence to a multivariate normal distribution holds:
\begin{equation}\label{eq:Q_N(X)}
\Big(Q_{k_j}(f_{n,j},\mathbf{X})\Big)_{j=1}^m \ConvD N(\mathbf{0},\Sigma).
\end{equation}
\item For a sequence $\mathbf{Z}=(Z_1,Z_2,\ldots)$  of i.i.d. standard normal random variables, the following joint weak convergence to a multivariate normal distribution holds:
\begin{equation}\label{eq:Q_N(Z)}
\Big(Q_{k_j}(f_{n,j},\mathbf{Z})\Big)_{j=1}^m \ConvD N(\mathbf{0},\Sigma).
\end{equation}
\end{enumerate}
\end{Thm}
\begin{Rem}
Condition (\ref{eq:second moment bound}) can be re-expressed as
\begin{equation}\label{eq:second moment bound practical}
\sup_n\E Q_{k_j}(f_{n,j},\mathbf{Z})^2=k_j! \sup_n \sum_{-\infty< i_1,\ldots,i_{k_j}< \infty} f_{n,j}(i_1,\ldots,i_{k_j})^2<\infty.
\end{equation}
\end{Rem}
\begin{Rem}
One can recover \citet{nourdin:peccati:2010:invariance} Theorem 1.2 from Theorem \ref{Thm:universality ext} by replacing $f_{n,j}(i_1,\ldots,i_{k_j})$ with $f_{n,j}(i_1,\ldots,i_{k_j})1_{1\le i_1,\ldots,i_{k_j}\le N_n}(i_1,\ldots,i_{k_j})$.
\end{Rem}
\begin{Rem}
In the one dimensional case: $m=1$, one can relax the assumption (\ref{eq:uniform bound third moment}) by $\sup_{i} \E|X_{i}|^{2+\delta}<\infty$ for any $\delta>0$. See Theorem 1.10 of \citet{nourdin:peccati:2010:invariance}.
\end{Rem}

\begin{proof}[Proof of Theorem \ref{Thm:universality ext}]
We need to prove that (\ref{eq:Q_N(Z)}) implies (\ref{eq:Q_N(X)}).
Define the $N_n$-truncated functions
\[
\tilde{f}_{n,j}(i_1,\ldots,i_{k_j})=f_{n,j}(i_1,\ldots,i_{k_j})1_{\{-N_n\le i_1\le N_n,\ldots,-N_n\le i_{k_j}\le N_n\}},\quad  j=1,\ldots,m.
\]
For  any $n\in \mathbb{Z}_+$, we can find $N_n$ large enough, so that for all $j=1,\ldots,m$,
\begin{equation}\label{eq:approx}
\E \left|Q_{k_j}(f_{n,j},\mathbf{Z})-Q_{k_j}(\tilde{f}_{n,j},\mathbf{Z})\right|^2=\E \left|Q_{k_j}(f_{n,j},\mathbf{X})-Q_{k_j}(\tilde{f}_{n,j},\mathbf{X})\right|^2= k_j! \|\tilde{f}_{n,j}-f_{n,j}\|_{\ell^2(\mathbb{Z}^{k_j})}^2\le \frac{1}{n}.
\end{equation}
Assume without loss of generality that $N_n\rightarrow\infty$ as $n\rightarrow\infty$. By (\ref{eq:Q_N(Z)})  and (\ref{eq:approx}), one has
\begin{equation*}
\Big(Q_{k_j}(\tilde{f}_{n,j},\mathbf{Z})\Big)_{j=1}^m  \ConvD N(\mathbf{0},\Sigma).
\end{equation*}
Using the original version of the universality result in  \citet{nourdin:peccati:2010:invariance} Theorem 1.2, one gets
\begin{equation}\label{eq:approx final}
\Big(Q_{k_j}(\tilde{f}_{n,j},\mathbf{X})\Big)_{j=1}^m \ConvD N(\mathbf{0},\Sigma).
\end{equation}
The conclusion (\ref{eq:Q_N(X)}) follows from (\ref{eq:approx}) and (\ref{eq:approx final}).
\end{proof}
\begin{Rem}
Using the same argument as in the preceding proof,  one can  eliminate the finiteness of $N_n$ in (\ref{eq:finite form}) in the following related universality results for homogeneous polynomial forms: (a) Theorem 1.12 of
\citet{nourdin:peccati:2010:invariance} concerning
for convergence to a $\chi^2$ distribution; (b) Theorem 3.4 of \citet{peccati:zheng:2014:universal} which is the counterpart of Theorem \ref{Thm:universality ext} here with $Z_i$'s being standardized Poisson random variables.
\end{Rem}

Theorem \ref{Thm:universality ext} gives rise to a practical criterion for the convergence (\ref{eq:Q_N(X)}).
We first introduce the discrete contraction operator:
for  $f\in \ell^p(\mathbb{Z}^p)$ and  $g\in \ell^q(\mathbb{Z}^q)$, $p,q\ge 2$, we define
\begin{equation}\label{eq:contraction def}
(f\star_r g)(i_1,\ldots,i_{p+q-2r})=\sum_{j_1,\ldots,j_r=-\infty}^{\infty}f(j_1,\ldots,j_r, i_1,\ldots,i_{p-r})g(j_1,\ldots,j_r, i_{p-r+1},\ldots,i_{p+q-2r})
\end{equation}
for $ r=0,\ldots,p\wedge q$,
where in the case $r=0$ it is understood as the tensor product.
\begin{Thm}\label{Thm:practical}
Let $\{f_{n,j}(\cdot),~ n\in \mathbb{Z}_+\}$ be a sequence of functions in $\ell^2(\mathbb{Z}^{k_j})$ satisfying (\ref{eq:second moment bound}), $j=1,\ldots,m$, where $k_j\ge 2$. Let $\Sigma$ be an $m\times m$
symmetric non-negative definite matrix whose each diagonal entry is positive, such that
\begin{equation}\label{eq:Sigma conv}
\Sigma(i,j)=\lim_{n\rightarrow\infty} \langle f_{n,i}, f_{n,j} \rangle,
\end{equation}
where $\langle \cdot, \cdot \rangle$ is defined in (\ref{eq:cov fun}).
Then the following are equivalent:
\begin{enumerate}
\item
For \emph{every} $\mathbf{X}=(X_{1},X_{2},\ldots)$ with $X_i$'s being independent random variables satisfying $\E X_{i}=0, \E X_{i}^2=1$ and $
\sup_{i} \E|X_{i}|^3<\infty$, we have the following joint weak convergence to normal:
\begin{equation}\label{eq:Q_kj to normal}
\Big(Q_{k_j}(f_{n,j},\mathbf{X})\Big)_{j=1}^m \ConvD N(\mathbf{0},\Sigma).
\end{equation}

\item
The following contractions are vanishing:
\begin{equation}\label{eq:contract cond clt}
\lim_{n\rightarrow\infty}\|f_{n,j}\star_r f_{n,j}\|_{2k_j-2r}=0,\quad \text{for all }r=1,\ldots,k_j-1\text{ and all }j=1,\ldots,m.
\end{equation}
where $\|\cdot\|_k$ denotes the discrete $L^2$ norm on $\ell^2(\mathbb{Z}^{k})$.
\end{enumerate}
\end{Thm}
\begin{proof}
By Theorem \ref{Thm:universality ext}, the statement 1 is equivalent to $\big(Q_{k_j}(f_{n,j},\mathbf{Z})\big)_{j=1}^m \ConvD N(\mathbf{0},\Sigma)$, where $\mathbf{Z}$ is a sequence of i.i.d.\ standard Gaussian variables. Note also that each
$Q_{k_j}(f_{n,j},\mathbf{Z})$ can be expressed as a $k_j$-tuple Wiener-It\^o integral with respect to Brownian motion. For  Wiener-It\^o integrals, joint convergence to the normal is equivalent to marginal convergence, and marginal convergence is equivalent to the contraction relations. More precisely,
 by applying Theorem 6.2.3 and 5.2.7 of \citet{nourdin:peccati:2012:normal}, one gets the equivalence to (\ref{eq:contract cond clt}).  See also Theorem  7.5 of \citet{nourdin:peccati:2010:invariance}.
\end{proof}
\begin{Rem}
We shall use the implication ``Statement 2 $\Rightarrow$ Statement 1'' of the preceding theorem in the sequel. As for the reversed implication, namely, ``Statement 1 $\Rightarrow$ Statement 2'', the stipulation ``For every'' is important here, as well as in Theorem \ref{Thm:universality ext}, because there are random variables $X_i$'s, for example Rademacher, that is $X_i=\pm 1$ with probability $1/2$ each, for which one may have convergence in (\ref{eq:Q_kj to normal}) even when (\ref{eq:contract cond clt}) does not hold (see \citet{nourdin:peccati:2010:invariance}, Section 1.6, p.1956).
\end{Rem}

\begin{Rem}
One may wonder if the universality result extends to a continuous setting, namely, when $Q_k(f_n)$ is replaced by a multiple integral on a Borel measure space $(A,\mathcal{A},\mu)$:
\[
I_k(f_n,\xi)=\int_{A^k}' f_n(x_1,\ldots,x_k) \xi(dx_1)\ldots \xi(dx_k),
\]
where $f\in L^2(A^k)$,  the prime $'$ indicates the exclusion of diagonals $x_p=x_q$, $p\neq q$, and $\xi(\cdot)$ is an independently scattered random measure with an atomless control measure $\mu(\cdot)$. Does $I_k(f_n,\xi)$ exhibits  a similar universality phenomenon? Namely, if $I_k(f_n,\xi)$ converges in distribution to normal for  a Gaussian $\xi(\cdot)$, does the convergence also hold for general class of $\xi(\cdot)$ with the same control measure $\mu(\cdot)$?
It is known that the law of $\xi(\cdot)$ has to be infinitely divisible and $\xi(\cdot)$ admits the decomposition:
\begin{equation}\label{eq:inf div decomp}
\xi(B)= G(B)+\int_{\mathbb{R}}\int_A u1_B(x) \widehat{N}(du,dx),
\end{equation}
where $G(\cdot)$ is a Gaussian random measure on $A$ and $\widehat{N}(\cdot)$ is an independent compensated Poisson random measure on $\mathbb{R}\times A$. See Section 5.3 of \citet{peccati:taqqu:2011:wiener} for more details.

One may think of adapting the approximation argument used in the proof of Theorem \ref{Thm:universality ext}  to the multiple integral case, which would involve partitioning the space $A$ into subsets of small measure. The problem is that unlike the Gaussian part, the Poisson part does not scale as $\mu(B)\rightarrow 0$. To see this in the simplest situation, take  $\xi(B)= \widehat{P}(B)$, where $\widehat{P}(\cdot)$ is a compensated Poisson random measure on $A$ with control measure $\mu(\cdot)$. Note that $\widehat{P}(B)+\mu(B)$  follows a Poisson distribution with mean $\mu(B)$.  Since its cumulants are all equal to $\mu(B)$ (see (3.1.5) of \citet{peccati:taqqu:2011:wiener}), and since the third moment of a centered random variable is equal to the third cumulant, one has  $\E (\widehat{P}(B))^3=\mu(B)$. This means that although we have the standardization
\begin{equation}\label{eq:standardization}
\E \left|\widehat{P}(B)/\sqrt{\mu(B)}\right|^2=1,
\end{equation}
we also have
\[
\lim_{\mu(B)\rightarrow 0} \E \left|\widehat{P}(B)/\sqrt{\mu(B)}\right|^3=\lim_{\mu(B)\rightarrow 0} \E \left|\widehat{P}(B)\right|^3 \mu(B)^{-3/2}\ge \lim_{\mu(B)\rightarrow 0}\E \widehat{P}(B)^3 \mu(B)^{-3/2}= \lim_{\mu(B)\rightarrow 0} \mu(B)^{-1/2}=\infty.
\]
This will violate  condition (\ref{eq:uniform bound third moment}) as the partition of $A$ becomes finer.
In fact, one can show that $\widehat{P}(B)/\sqrt{\mu(B)}\rightarrow 0$ in probability as $\mu(B)\rightarrow 0$, which means, in view of (\ref{eq:standardization}), that the uniform integrability of $|\widehat{P}(B)/\sqrt{\mu(B)}|^2$ fails. For further insights, see \citet{rotar:1979:limit}.
\end{Rem}

\section{Application: boundary between short and long memory}\label{Sec:boundary}
\subsection{The setting}\label{Sec:setting}

\citet{bai:taqqu:2013:3-generalized} considered the following \emph{discrete chaos processes}:
\begin{equation}\label{eq:discrete chaos}
X(n)=\sum_{-\infty<i_1,\ldots,i_k<n} a(n-i_1,\ldots,n-i_k) \epsilon_{i_1}\ldots \epsilon_{i_k},
\end{equation}
where $k\ge 2$, $a(\cdot): \mathbb{Z}_+^k\rightarrow \mathbb{R}$ is symmetric and vanishes on the diagonals, and $\epsilon_i$'s are i.i.d.\ random variables with mean $0$ and variance $1$.  Note that $\E X(n)=0$.

In particular,  \citet{bai:taqqu:2013:3-generalized} studied limit theorems for normalized partial sum process of $X(n)$:
\[
Y_N(t):=\frac{1}{A(N)} \sum_{n=1}^{[Nt]} X(n),
\]
where $[\cdot]$ means integer part, and $A(N)$ is a suitable normalization factor.
Depending on the behavior of $a(\cdot)$, the stationary process $X(n)$ may exhibit short or long memory.

As shown in \citet{bai:taqqu:2013:3-generalized}, in the short memory case, namely when  the coefficient in (\ref{eq:discrete chaos}) satisfies
the summability condition
\begin{equation}\label{eq:SRD summability}
\sum_{n=1}^\infty \sum_{0<i_1,\ldots,i_k<\infty}\Big|a(i_1,\ldots,i_k)a(i_1+n,\ldots,i_k+n)\Big|<\infty,
\end{equation}
and $\E |\epsilon_i|^{2+\delta}<\infty$ for some $\delta>0$,
 the following central limit convergence as $N\rightarrow\infty$ holds:
\begin{equation}\label{eq:clt}
\frac{1}{N^{1/2}} \sum_{n=1}^{[Nt]} X(n) \Rightarrow  \sigma B(t)
\end{equation}
for some $\sigma\ge 0$, where  $B(t)$ is a standard Brownian motion.

In the long memory case,  assume that
\begin{equation}\label{eq:a(.)}
a(\cdot)=g(\cdot)L(\cdot)1_{D^c},
\end{equation}
where
\begin{equation}\label{eq:D^c}
D^c:=\{(i_1,\ldots,i_k):i_p\neq i_q \text{ for }p\neq q\}
\end{equation}
guarantees that $a(\cdot)$ vanishes on the diagonals. The function
 $L(\cdot):\mathbb{Z}_+^k \rightarrow \mathbb{R}$ satisfies\footnote{In \citet{bai:taqqu:2013:3-generalized} eq.\ (25), $L(\cdot)$ is assumed to satisfy a slightly  weaker condition than (\ref{eq:L condition}), that is, $\lim_{N\rightarrow\infty}L([N \mathbf{x}]+\mathbf{B}(N))=1$ for any $\mathbf{x}\in \mathbb{R}_+^k$ and any bounded sequence $\mathbf{B}(N)$ in $\mathbb{Z}_+^k$ instead of $\lim_{\|\mathbf{x}\|\rightarrow\infty}L(\mathbf{x})=0$. Note that $L([N \mathbf{x}]+\mathbf{B}(N))$, $N\rightarrow\infty$, lets the argument increase in a specific band in the first quadrant, whereas $L(\mathbf{x})$, $\|\mathbf{x}\|\rightarrow\infty$, allows $\mathbf{x}$ to increase in an arbitrary way in the first quadrant. Here for simplicity we just assume (\ref{eq:L condition}), while the results stated here also hold under the  weaker condition.}
\begin{equation}\label{eq:L condition}
\lim_{|\mathbf{i}|\rightarrow\infty} L(\mathbf{i})=1,
\end{equation}
and $g(\cdot): \mathbb{R}^k \rightarrow \mathbb{R}$ is the so-called \emph{generalized Hermite kernel of Class (L)}.


\begin{Def} \label{Def:GHK L}
A nonzero a.e.\ continuous function $g(\cdot): \mathbb{R}^k \rightarrow \mathbb{R}$ is called
a generalized Hermite kernel of Class (L) (GHK(L)) if it satisfies
\begin{enumerate}
\item $g(\cdot)$ is homogeneous with exponent $\alpha$, namely, $g(\lambda \mathbf{x})=\lambda^{\alpha} g(\mathbf{x})$, for all $\lambda>0$, where
\begin{equation}\label{eq:alpha range}
\alpha\in \left(-\frac{k+1}{2},-\frac{k}{2}\right);
\end{equation}
\item  The function $g(\cdot)$ satisfies the bound
\begin{equation}\label{eq:g bound g*}
|g(\mathbf{x})|\le g^*(\mathbf{x}):= c\sum_{j=1}^m  x_1^{\gamma_{j1}} \ldots x_k^{\gamma_{jk}},
\end{equation}
 with the constant $c>0$, $-1<\gamma_{jl}<-1/2$ and $\sum_{l=1}^k \gamma_{jl}=\alpha$ for all $l=1,\ldots,m$.
\end{enumerate}
\end{Def}
\bigskip

If $g$ is a GHK(L), the following constant is well-defined (the integral is absolutely integrable)
\begin{equation}\label{eq:C_g}
C_g=\int_{\mathbb{R}_+^k} g(x_1,\ldots,x_k)g(1+x_1,\ldots,1+x_k) dx_1\ldots dx_k,
\end{equation}
and $C_g>0$ always (Remark 3.6 of \citet{bai:taqqu:2013:3-generalized}).
Under this setup, Theorem 6.5 of \citet{bai:taqqu:2013:3-generalized} showed that  as $N\rightarrow\infty$,
\begin{equation}\label{eq:nclt}
\frac{1}{N^H} \sum_{n=1}^{[Nt]} X(n) \Rightarrow  \int_{\mathbb{R}}' ~\int_0^t g(s_1-x_1,\ldots,s_k-x_k) 1_{\{s_1>x_1,\ldots,s_k>x_k\}}~ W(dx_1)\ldots W(dx_k),
\end{equation}
where $W(\cdot)$ is the Brownian random measure, the prime $'$ indicates the exclusion of the diagonals $x_p=x_q$, $p\neq q$, and
\[
H=\alpha+\frac{k}{2}+1.
\]
The limit in (\ref{eq:nclt}) was called a \emph{generalized Hermite process} which  generalizes the Hermite process (see, e.g., \citet{dobrushin:major:1979:non} and \citet{taqqu:1979:convergence}) which corresponds to the special case $g(\mathbf{x})=x_1^{\alpha/k}\ldots x_k^{\alpha/k}$.

There is, however, a boundary case which the limit theorems (\ref{eq:clt}) and (\ref{eq:nclt})  did not cover. This boundary case is as follows: set as in the long memory case
\begin{equation}\label{eq:a coef}
a(\cdot)=g(\cdot)L(\cdot)1_{D^c},
\end{equation}
where $D^c$ is as in (\ref{eq:D^c}), $L(\cdot)$ is as in (\ref{eq:L condition}), and $g$ is a function satisfying the assumptions in Definition \ref{Def:GHK L} except that instead of assuming (\ref{eq:alpha range}), the homogeneity exponent is set as $\alpha=-\frac{k+1}{2}$.

\begin{Rem}\label{Rem:short memory exponent}
Note that if $\alpha<-\frac{k+1}{2}$, we are in the short memory regime. Indeed
 Proposition 5.4 of \citet{bai:taqqu:2014:4-convergence} showed that $\alpha<-\frac{k+1}{2}$ implies (\ref{eq:SRD summability}), and thus (\ref{eq:clt}) holds. So (\ref{eq:boundary alpha}) is exactly the boundary case between short and long memory.
\end{Rem}
\subsection{Statement of the limit theorems}
Let throughout
$\Rightarrow$ denote weak convergence in Skorohod space $D[0,1]$ with uniform metric.
We shall show  by the criterion formulated in  Theorem \ref{Thm:practical}, that a central limit theorem holds with an extra logarithmic factor in the normalization:
\begin{Thm}[Nonlinear case]\label{Thm:boundary k>=2}
Let
\[
X(n)=\sum_{-\infty<i_1,\ldots,i_k<n} a(n-i_1,\ldots,n-i_k) \epsilon_{i_1}\ldots \epsilon_{i_k}
\]
 as in (\ref{eq:discrete chaos}) with $k\ge 2$ and the coefficient $a(\cdot)$ specified as in (\ref{eq:a coef}) where
\begin{equation}\label{eq:boundary alpha}
\alpha=-\frac{k+1}{2}.
\end{equation}
Assume also that $\E |\epsilon_i|^{3}<\infty$ and $C_g>0$.
Then
\begin{equation*}
Y_N(t):=\frac{1}{\sqrt{N\ln N}}\sum_{n=1}^{[Nt]} X(n) \Rightarrow \sigma B(t)
\end{equation*}
where $\sigma=\sqrt{2C_g}$, and $B(t)$ is a standard Brownian motion.
\end{Thm}
\begin{Rem}
Theorem \ref{Thm:boundary k>=2} may be compared to a similar boundary case of limit theorems for nonlinear transform of long-memory Gaussian noise first considered in \citet{breuer:major:1983:central} Theorem $1'$. The proof there was done by a method of moments. See also  \citet{breton:nourdin:2008:error} who gave an alternative proof using the Malliavin calculus.
\end{Rem}
Note that to apply Theorem \ref{Thm:practical}, the process $X(n)$ in (\ref{eq:discrete chaos}) needs to have order $k\ge 2$. For completeness, we state also the corresponding result for linear process, namely, the case $k=1$ in Theorem \ref{Thm:boundary k>=2}, though the limit theorem for  linear process is classical (see,e.g., \citet{davydov:1970:invariance}).
\begin{Thm}[Linear case]\label{Thm:linear}
Let
\[
X(n)=\sum_{-\infty<i<n} a(n-i)\epsilon_i,
\]
where $a(n) = L(n) n^{-1}$ as $n\rightarrow\infty$, and let
 $L(n)\rightarrow c\neq 0$, and the i.i.d.\ standardized noise $\epsilon_i$'s satisfy $\E|\epsilon_i|^{2+\delta}<\infty$ for some $\delta>0$.
Then as $N\rightarrow\infty$,
\begin{equation*}
Y_N(t):=\frac{1}{\sqrt{N}~\ln N}\sum_{n=1}^{[Nt]} X(n) \Rightarrow \sigma B(t)
\end{equation*}
where $\sigma=\sqrt{2}|c|$, and $B(t)$ is a standard Brownian motion.
\end{Thm}
\subsection{Proof of Theorem \ref{Thm:boundary k>=2}}\label{Sec:proof boundary k>=2}

We first  compute the asymptotic variance of the sum.
\begin{Lem}\label{Lem:var}
Let $X(n)$ be given as in (\ref{eq:discrete chaos}) with the coefficient specified as in (\ref{eq:a coef}) and $\alpha$  as in (\ref{eq:boundary alpha}). Then $C_g$ defined in (\ref{eq:C_g}) is non-negative. If $C_g>0$, then  as $N\rightarrow\infty$
\begin{equation*}
\E\left[ \sum_{n=1}^N X(n)\right]^2 \sim 2C_g N\ln N.
\end{equation*}
If  $C_g=0$, then
\begin{equation}\label{eq:var degen}
\E\left[ \sum_{n=1}^N X(n)\right]^2=o(N\ln N).
\end{equation}
\end{Lem}
\begin{proof}
Assume for simplicity $L(\cdot)=1$, and it is easy to extend the following arguments to the general case.
First, since $g(\cdot)$ is homogeneous with exponent $\alpha=-k/2-1/2$  by (\ref{eq:boundary alpha}), one can write
\begin{align*}
\gamma(n):&=\E X(n)X(0)= \sum_{0<i_1,\ldots,i_k<\infty} g(i_1,\ldots,i_k)g(i_1+n,\ldots,i_k+n)1_{D^c}\left(i_1,\ldots,i_k\right)\\
&= n^{-1} \sum_{0<i_1,\ldots,i_k<\infty}g\left(\frac{i_1}{n},\ldots,\frac{i_k}{n}\right)g\left(\frac{i_1}{n}+1,\ldots,\frac{i_k}{n}+1\right)  1_{D^c}\left(i_1,\ldots,i_k\right)  n^{-k}\\
&=n^{-1}\int_{\mathbb{R}_+^k}g\left(\frac{[nx_1]+1}{n},\ldots,\frac{[nx_k]+1}{n}\right)g\left(\frac{[nx_1]+1}{n}+1,\ldots,\frac{[nx_k]+1}{n}+1\right)  1_{D^c}\left([nx_1],\ldots,[nx_k]\right) dx_1\ldots dx_k\\
&=: n^{-1}C_n(g).
\end{align*}
Because the bounding function $g^*$ in Definition \ref{Def:GHK L} is decreasing in  every variable,  the absolute of the integrand above is bounded by
\[
 g^*\left(x_1,\ldots,x_k\right)g^*\left(x_1+1,\ldots,x_k+1\right) = c^2\sum_{j_1,j_2=1}^m  x_1^{\gamma_{j_1,1}}(x_1+1)^{\gamma_{j_2,1}}  \ldots x_k^{\gamma_{j_1,k}}(x_k+1)^{\gamma_{j_2,k}}
\]
which is integrable on $\mathbb{R}_+^k$ because all $\gamma_{p,q}\in(-1, -1/2)$ and
\[
\int_\mathbb{R_+} x^{\gamma}(x+1)^{\gamma'} dx <\infty \quad\text{ for any $-1<\gamma,\gamma'<-1/2$.}
\]
Since $g$ is assumed to be a.e.\ continuous, by the Dominated Convergence Theorem, as $n\rightarrow\infty$ we have
\[
C_n(g)\rightarrow C_g:=\int_{\mathbb{R}^k}g\left(x_1,\ldots,x_k\right)g\left(x_1+1,\ldots,x_k+1\right) dx_1\ldots dx_k.
\]
Hence  when $C_n\neq 0$, one has when $n>0$
\[
\gamma(n)\sim n^{-1} C_g,
\]
and when $C_n=0$, one has
\[
\gamma(n)= o(n^{-1}).
\]

We shall use the fact that if $a_n\sim n^{-1}$ as $n\rightarrow\infty$, then $\sum_{n=1}^N a_n \sim \ln N$
as $N\rightarrow\infty$. So when $C_g\neq 0$, one has
\begin{align*}
\E\left[ \sum_{n=1}^N X(n)\right]&=\sum_{n_1,n_2=1}^N \gamma(n_1-n_2)=N\sum_{n=-N+1}^{N-1}\gamma(n)- \sum_{n=-N+1}^{N-1}|n|\gamma(n)\sim 2C_gN\ln N.
\end{align*}
Note that since $\gamma(n)\sim n^{-1} C_g$, the term $ \sum_{n=-N+1}^{N-1}|n|\gamma(n)\sim 2C_g N$  and is thus negligible.

The preceding asymptotic equivalence also shows that if $C_g\neq 0$ then $C_g>0$ because the variance is non-negative.

If $C_g=0$,  following similar lines of argument, one gets  ({\ref{eq:var degen}}).

\end{proof}

\begin{Lem}\label{Lem:bound diff}
Define the mapping $(\cdot,\cdot)_0: \mathbb{R}^2\rightarrow \mathbb{R}$ as
\[
(x_1,x_2)_0=\begin{cases}
  |x_1-x_2| &\text{ if } x_1\neq x_2;\\
  1 &\text{ if } x_1=x_2=x.
 \end{cases}
\]
For $-1<\gamma_1,\gamma_2<-1/2$ and $n_1,n_2\in \{1,2,\ldots\}$, we have  for some constant $C>0$ not depending on $n_1,n_2$ that
\[\sum_{p\in \mathbb{Z}} (n_1-p)_+^{\gamma_1} (n_2-p)_+^{\gamma_2}\le C (n_1,n_2)_0^{\gamma_1+\gamma_2+1}.
\]
\end{Lem}
\begin{proof}
For the case $n_1=n_2=n$, choose $C=\sum_{p<n} (n-p)^{\gamma_1+\gamma_2}<\infty$ since $\gamma_1+\gamma_2<-1$. When $n_1\neq n_2$, suppose that $n_1< n_2$. Then
\begin{align*}
\sum_{p\in \mathbb{Z}} (n_1-p)_+^{\gamma_1} (n_2-p)_+^{\gamma_2}&=\sum_{p=1}^\infty p^{\gamma_1}(n_2-n_1+p)^\gamma_2\le \int_{0}^\infty x^{\gamma_1}(n_2-n_1+x)^{\gamma_2} dx\\& =(n_2-n_1)^{\gamma_1+\gamma_2+1}\int_{0}^\infty y^{\gamma_1}(1+y)^{\gamma_2} dy,
\end{align*}
where the integral converges.
\end{proof}

The following simple fact will be used.
\begin{Lem}\label{Lem：sum restriction}
Suppose that $\gamma_j<-1/2$ for all $j=1,\ldots,k$, $k\ge 2$, and $\gamma_1+\ldots+\gamma_k\ge -k/2-1/2$.  Then
\[
-\frac{r}{2}-\frac{1}{2}<\gamma_1+\ldots+\gamma_r<-\frac{r}{2}\quad \text{for all $r=1,\ldots,k-1$.}
\]
In addition, each $\gamma_j>-1$, $j=1,\ldots,k$.
\end{Lem}
\begin{proof}
The inequality $\gamma_1+\ldots+\gamma_r<-\frac{r}{2}$ is obvious. For the other inequality,
suppose that $\gamma_1+\ldots+\gamma_r\le -r/2-1/2$ for some $r\in \{1,\ldots,k\}$. Because $\gamma_{r+1},\ldots,\gamma_{k}<-1/2$, we get the contradiction: $\gamma_1+\ldots+\gamma_k<-r/2-1/2-(k-r)/2=-k/2-1/2$.

Then we show by contradiction  that each $\gamma_j>-1$. Suppose, e.g.,  $\gamma_k\le -1$. By what was just proved, one has
 $\gamma_1+\ldots+\gamma_{k-1}<-(k-1)/2$. Thus by adding $\gamma_k\le -1$, one gets $\gamma_1+\ldots+\gamma_{k}<-k/2-1/2$, which contradicts the assumption.
\end{proof}

We need the following lemma, which  is a consequence of Corollary 1.1 (b) of \citet{terrin:taqqu:1991:power}.
\begin{Lem}\label{Lem:bound power circular}
If $\alpha_1,\ldots,\alpha_m$, $m\ge 2$, satisfy
\begin{equation}\label{eq:alpha exponent condition}
\alpha_1,\ldots,\alpha_n>-1, \quad \sum_{i=1}^m \alpha_i +m>1,
\end{equation}
then for any $c>0$
\[
\int_{[0,c]^m}  |x_1-x_2|^{\alpha_1}|x_2-x_3|^{\alpha_2}\ldots |x_{m-1}-x_m|^{\alpha_{m-1}}|x_m-x_1|^{\alpha_m}     dx_1\ldots dx_m < \infty.
\]
\end{Lem}

We need also the following hypercontractivity inequality for proving tightness in $D[0,1]$ (Proposition 5.2 of \citet{bai:taqqu:2013:3-generalized})
\begin{Lem}\label{Lem:Hypercontract}
Suppose that $h\in \ell^2(\mathbb{Z}^k)$ vanishing on the diagonals.
Let $X=\sum_{\mathbf{i}\in \mathbb{Z}^k}h(\mathbf{i}) \prod_{p=1}^k\epsilon_{i_p}$, $k\ge 1$.
If for some $p'>p>2$, $\E |\epsilon_i|^{p'}<\infty$, then one has for some constant $c_{p,k}>0$ which does not depend on $h$ that
\begin{equation*}
\E [|X|^p]^{1/p} \le c_{p,k} \E [|X|^2]^{1/2}.
\end{equation*}
\end{Lem}

\begin{proof}[Proof of Theorem \ref{Thm:boundary k>=2}]
Let $C>0$ be a constant whose value can change from line to line.
We first show that the finite-dimensional distributions of $Y_N(t)$ converges to those of $\sigma B(t)$ using  Theorem \ref{Thm:practical}. First,
the convergence of the covariance structure of $Y_N(t)$ to that of $\sigma B(t)$ follows from Lemma \ref{Lem:var} the fact that for $s\le t$ we have
\[
\E Y_N(t)Y_N(s)=\frac{1}{2}\left[\E Y_N(t)^2 +\E Y_N(s)^2- \E(Y_N(t)-Y_N(s))^2\right]\sim \frac{1}{2}\left[\E Y_N(t)^2 +\E Y_N(s)^2- \E Y_N(t-s)^2\right]
\]
as $N\rightarrow\infty$,
since $X(n)$ is stationary. We now check the contraction conditions (\ref{eq:contract cond clt}).
For simplicity we set $L(\cdot)=1$ and $t=1$.
We can write
\[
Y_N(1)=\sum_{-\infty<i_1,\ldots,i_k<+\infty} ~ f_N(i_1,\ldots,i_k)  ~ \epsilon_{i_1}\ldots \epsilon_{i_k}
\]
where
\begin{equation}\label{eq:f_N}
f_N(i_1,\ldots,i_k)=\frac{1}{\sqrt{N\ln N}}\sum_{n=1}^{N} g\left(n-i_1,\ldots, n-i_k\right) 1_{D^c\cap \{i_1<n,\ldots,i_k<n\}}.
\end{equation}

To simplify notation, we set
\[\mathbf{p}=(p_1,\ldots,p_r), \qquad \mathbf{q}=(q_1,\ldots,q_{k-r}),
\]
\[\mathbf{i}_1=(i_1,\ldots,i_{k-r}),\quad \mathbf{i}_2=(i_{k-r+1},\ldots,i_{2k-2r}),\quad \mathbf{i}=(\mathbf{i}_1,\mathbf{i}_2),
\]
and let $\mathbf{1}$ stand for a vector of $1$'s of suitable dimension. We also use the convention that $\mathbf{x}^\mathbf{a}=x_1^{a_1}\ldots x_n^{a_n}$ if $\mathbf{x}=(x_1,\ldots,x_n)$ and $\mathbf{a}=(a_1,\ldots,a_n)$. Let $(\Sigma\mathbf{x})=x_1+\ldots+x_n$ if $\mathbf{x}=(x_1,\ldots,x_n)$.

Set  $g^*(\cdot)$ be as in Definition \ref{Def:GHK L} which we write by splitting
$\mathbf{x}=(\mathbf{x}_1,\mathbf{x}_2)$, where $\mathbf{x}_1\in \mathbb{R}_+^{r}$ and $\mathbf{x}_2\in \mathbb{R}_+^{k-r}$:
\begin{equation}\label{eq:g*(x1,x2)}
g^*(\mathbf{x}_1,\mathbf{x}_2)=c\sum_{j=1}^m    \mathbf{x}_1^{\mathbd{\beta}_j} \mathbf{x}_2^{\mathbd{\eta}_j}, \quad \mathbd{\beta}_j=(\gamma_{j1},\ldots,\gamma_{jr}),\quad \mathbd{\eta}_j=(\gamma_{j,r+1},\ldots,\gamma_{jk}),
\end{equation}
so that
\begin{equation}\label{eq:sum beta eta}
\sum_{i=1}^r \beta_{ji}+\sum_{i=1}^{k-r}\eta_{ji}=\sum_{i=1}^k \gamma_{ji}=\alpha,
\end{equation}
which we write simply as $\sum \mathbd{\beta}+\sum \mathbd{\eta}=\sum \mathbd{\gamma}=\alpha$.
 For convenience, if some component $x_j$ of $\mathbf{x}$ is negative, we set  $\mathbf{x}^a=0$ and hence $g^*(\mathbf{x})=0$.  Then in view of (\ref{eq:f_N}), (\ref{eq:contraction def}) and (\ref{eq:g bound g*}),
\begin{align*}
|(f_N\star_r f_N) (\mathbf{i})|&\le \frac{1}{N\ln N} \sum_{n_1,n_2=1}^{N} \sum_{\mathbf{p}} g^*(n_1\mathbf{1}-\mathbf{p},n_1\mathbf{1}-\mathbf{i}_1) g^*(n_2\mathbf{1}-\mathbf{p},n_2\mathbf{1}-\mathbf{i}_2)\\
&= \frac{c^2}{N\ln N} \sum_{n_1,n_2=1}^{N} \sum_{j_1,j_2=1}^m (n_1\mathbf{1}-\mathbf{i}_1)^{\mathbd{\eta}_{j_1}}(n_2\mathbf{1}-\mathbf{i}_2)^{\mathbd{\eta}_{j_2}}   \sum_{\mathbf{p}}  (n_1\mathbf{1}-\mathbf{p})^{\mathbd{\beta}_{j_1}}(n_2\mathbf{1}-\mathbf{p})^{\mathbd{\beta}_{j_2}},
\end{align*}
by using (\ref{eq:g*(x1,x2)}).
By Lemma \ref{Lem:bound diff}, we have for the last sum,
\[
\sum_{\mathbf{p}}  (n_1\mathbf{1}-\mathbf{p})^{\mathbd{\beta}_{j_1}}(n_2\mathbf{1}-\mathbf{p})^{\mathbd{\beta}_{j_2}}=\sum_{p_1,\ldots,p_r} \prod_{u=1}^r (n_1-p_u)^{\gamma_{j_1,u}} \prod_{v=1}^r (n_1-p_v)^{\gamma_{j_2,v}} \le C (n_1,n_2)_0^{(\Sigma\mathbd{\beta}_{j_1})+(\Sigma\mathbd{\beta}_{j_2})+r}.
\]
Hence
\begin{align}
\|f_N\star_r f_N \|_{2k-2r}^2&=
\sum_{\mathbf{i}}\left[(f_N\star_r f_N) (\mathbf{i})\right]^2\notag\\
 &\le
 \frac{C}{N^2(\ln N)^2} \sum_{\mathbf{i}} \left(\sum_{n_1,n_2=1}^{N} \sum_{j_1,j_2=1}^m(n_1,n_2)_0^{(\Sigma\mathbd{\beta}_{j_1})+(\Sigma\mathbd{\beta}_{j_2})+r} (n_1\mathbf{1}-\mathbf{i}_1)^{\mathbd{\eta}_{j_1}}(n_2\mathbf{1}-\mathbf{i}_2)^{\mathbd{\eta}_{j_2}}   \right)^2 \notag \\
 &=  \frac{C}{N^2(\ln N)^2}\sum_{j_1,j_2,j_3,j_4=1}^m \sum_{n_1,n_2,n_3,n_4=1}^{N} (n_1,n_2)_0^{(\Sigma\mathbd{\beta}_{j_1})+(\Sigma\mathbd{\beta}_{j_2})+r} (n_3,n_4)_0^{(\Sigma\mathbd{\beta}_{j_3})+(\Sigma\mathbd{\beta}_{j_4})+r} \notag \\& \qquad \times \sum_{\mathbf{i}_1} (n_1\mathbf{1}-\mathbf{i}_1)^{\mathbd{\eta}_{j_1}}(n_3\mathbf{1}-\mathbf{i}_1)^{\mathbd{\eta}_{j_3}} \sum_{\mathbf{i}_2}(n_2\mathbf{1}-\mathbf{i}_2)^{\mathbd{\eta}_{j_2}}(n_4\mathbf{1}-\mathbf{i}_2)^{\mathbd{\eta}_{j_4}} \notag \\
 &\le \frac{C}{N^2(\ln N)^2}\sum_{j_1,j_2,j_3,j_4=1}^m \sum_{n_1,n_2,n_3,n_4=1}^{N} (n_1,n_2)_0^{(\Sigma\mathbd{\beta}_{j_1})+(\Sigma\mathbd{\beta}_{j_2})+r} (n_3,n_4)_0^{(\Sigma\mathbd{\beta}_{j_3})+(\Sigma\mathbd{\beta}_{j_4})+r} \notag\\& \qquad\times
(n_1,n_3)_0^{(\Sigma\mathbd{\eta}_{j_1})+(\Sigma\mathbd{\eta}_{j_3})+k-r} (n_2,n_4)_0^{(\Sigma\mathbd{\eta}_{j_2})+(\Sigma\mathbd{\eta}_{j_4})+k-r}  \label{eq:f_N contract bound}
\end{align}
where we have applied again Lemma \ref{Lem:bound diff} to get the last inequality.
Note that  if one adds up the power exponents in the last expression, one gets
\begin{equation}\label{eq:sum of expo}
(\Sigma\mathbd{\beta}_{j_1})+(\Sigma\mathbd{\eta}_{j_1})+(\Sigma\mathbd{\beta}_{j_2})+(\Sigma\mathbd{\eta}_{j_2})+(\Sigma\mathbd{\beta}_{j_3})+(\Sigma\mathbd{\eta}_{j_3})+(\Sigma\mathbd{\beta}_{j_4})+(\Sigma\mathbd{\eta}_{j_4})+2k=4\alpha+2k=-2,
\end{equation}
by (\ref{eq:sum beta eta}), where the last equality of (\ref{eq:sum of expo}) is due to assumption (\ref{eq:boundary alpha}).

Note also that by Lemma \ref{Lem：sum restriction}, we have for $r\in\{1,\ldots,k-1\}$ that
\[
-\frac{r}{2}-\frac{1}{2}<(\Sigma\mathbd{\beta}_{j_1}),(\Sigma\mathbd{\beta}_{j_2}),(\Sigma\mathbd{\beta}_{j_3}),(\Sigma\mathbd{\beta}_{j_4})<-\frac{r}{2},
\]
and
\[
 -\frac{k-r}{2}-\frac{1}{2}<
(\Sigma\mathbd{\eta}_{j_1}),(\Sigma\mathbd{\eta}_{j_3}),(\Sigma\mathbd{\eta}_{j_2}),(\Sigma\mathbd{\eta}_{j_4})<-\frac{k-r}{2}.
\]
Let $\alpha_1=(\sum \mathbb{\beta}_{j_1})+(\sum \mathbb{\beta}_{j_2})+r$ be the exponent of $(n_1,n_2)_0$ in (\ref{eq:f_N contract bound}). Then
\[
-1=-r/2-1/2-r/2-1/2+r<\alpha_1<-r/2-r/2+r=-r+r=0.
\]
Define similarly $\alpha_2,\alpha_3,\alpha_4$ for the other exponents in (\ref{eq:f_N contract bound}), which all lie strictly between $-1$ and $0$.
Hence,
the convergence
\begin{equation}\label{eq:contract in proof}
\lim_{N\rightarrow\infty}\|f_N\star_r f_N \|_{2k-2r}^2=0, \quad r=1,\ldots,k-1,
\end{equation}
will follow if one shows that
\begin{equation}\label{eq:goal bounded}
\sup_N~
N^{-2} \sum_{n_1,n_2,n_3,n_4=1}^{N} (n_1,n_2)_0^{\alpha_1} (n_2,n_3)_0^{\alpha_2}
(n_3,n_4)_0^{\alpha_3} (n_4,n_1)_0^{\alpha_4}<\infty,
\end{equation}
where by (\ref{eq:sum of expo})
\begin{equation}\label{eq:alpha j cond}
-1<\alpha_j<0,~j=1,\ldots 4, \quad \alpha_1+\alpha_2+\alpha_3+\alpha_4=-2.
\end{equation}

Let's consider first the sum in (\ref{eq:goal bounded}) over only distinct $n_1,\ldots,n_4$ (we use the prime $'$ to indicate that the sum does not include the diagonals).
In this case,
\begin{align*}
&\sum_{1\le n_1,n_2,n_3,n_4\le N}' \left|\frac{n_1}{N}-\frac{n_2}{N}\right|^{\alpha_1} \left|\frac{n_2}{N}-\frac{n_3}{N}\right|^{\alpha_2}
\left|\frac{n_3}{N}-\frac{n_4}{N}\right|^{\alpha_3} \left|\frac{n_4}{N}-\frac{n_1}{N}\right|^{\alpha_4} N^{-4}
=\int\left|\frac{[Nx_1]-[Nx_2]}{N}\right|^{\alpha_1}\times \\& \left|\frac{[Nx_2]-[Nx_3]}{N}\right|^{\alpha_2}
\left|\frac{[Nx_3]-[Nx_4]}{N}\right|^{\alpha_3} \left|\frac{[Nx_4]-[Nx_1]}{N}\right|^{\alpha_4} I\{N^{-1}\le x_i\le 1+N^{-1},~[Nx_i]\neq [Nx_j], \forall i\neq j\}d\mathbf{x}.
\end{align*}
Note that for any $x,y>0$, one has that $|[Nx]-[Ny]|=n$ implies that $|Nx-Ny|\le n+1$ which  implies $|x-y|\le (n+1)/N$, for $n\ge 0$. Then
since each $\alpha<0$, we get
\begin{align*}
&\sup_N\left|\frac{[Nx]-[Ny]}{N}\right|^{\alpha} |x-y|^{-\alpha}I\{[Nx]\neq [Ny]\}\le
\sup_{|[Nx]-[Ny]|=n, n\in \mathbb{Z}_+} \left(\frac{n}{N}\right)^{\alpha} \left(\frac{n+1}{N}\right)^{-\alpha}= \sup_{n\in \mathbb{Z}_+} \left(\frac{n+1}{n}\right)^{-\alpha} = 2^{-\alpha}.
\end{align*}
Hence the the sum in (\ref{eq:goal bounded}) over distinct $n_1,\ldots,n_4$ is bounded by
\[
C\int_{[0,2]^4}  |x_1-x_2|^{\alpha_1}|x_2-x_3|^{\alpha_2} |x_{3}-x_4|^{\alpha_{3}}|x_4-x_1|^{\alpha_4}     dx_1dx_2dx_3 dx_4 ,
\]
which is finite due to Lemma \ref{Lem:bound power circular}.

Consider now the the sum in (\ref{eq:goal bounded}) over $n_1,\ldots,n_4$ with only  three of them distinct. Let, for example, $n_1=n_4$, and we need to show that the following
\begin{align*}
&\sup_{N} N^{-2}\sum_{1\le n_1,n_2,n_3\le N}' |n_1-n_2|^{\alpha_1}|n_2-n_3|^{\alpha_2}
|n_3-n_1|^{\alpha_3}=\\
& \sup_N N^{1+\alpha_1+\alpha_2+\alpha_3}\sum_{1\le n_1,n_2,n_3\le N}' \left|\frac{n_1}{N}-\frac{n_2}{N}\right|^{\alpha_1}\left|\frac{n_2}{N}-\frac{n_3}{N}\right|^{\alpha_2}
\left|\frac{n_3}{N}-\frac{n_1}{N}\right|^{\alpha_3} N^{-3}<\infty.
\end{align*}
Note that (\ref{eq:alpha j cond}) entails that $-2<\alpha_1+\alpha_2+\alpha_3<-1$. Then $ N^{1+\alpha_1+\alpha_2+\alpha_3}\rightarrow 0$ as $N\rightarrow\infty$, and  the boundedness of the multiple sum can be established similarly as above using integral approximation and Lemma \ref{Lem:bound power circular}.

If the sum in (\ref{eq:goal bounded}) is over $n_1,\ldots,n_4$ with only two or less of them distinct, the boundedness is easily established through bounding all the summands by one constant,  because we have the factor $N^{-2}$.

So (\ref{eq:goal bounded}) holds and thus (\ref{eq:contract in proof}) holds, and the convergence of finite-dimensional distributions is proved.

Now we show tightness.  By Lemma \ref{Lem:Hypercontract}, one can choose $p\in(2,3)$, so that by Lemma \ref{Lem:var} if $0<s<t<1$, one has for $N$ large enough,
\begin{align*}
\E |Y_N(t)-Y_N(s)|^p \le&  C [\E |Y_N(t)-Y_N(s)|^2]^{p/2}\le C\left[\frac{[Nt]-[Ns]}{N} \cdot \frac{\ln([Nt]-[Ns])}{\ln N} \right]^{p/2}
\\\le& C  \left[\frac{[Nt]-[Ns]}{N}\right]^{p/2-\delta},
\end{align*}
where $\delta>0$ is small enough so that $p/2-\delta>1$. The last inequality is true because $\ln x$ is slowly varying as $x\rightarrow\infty$ and so one applies the Potter's bound (see e.g., equation (2.3.6) of \citet{giraitis:koul:surgailis:2009:large}).
Note that $F_N(t):=[Nt]/N$ is a non-decreasing right continuous function on $[0,1]$  and that $F_N$ converges uniformly to $F(t):=t$ as $N\rightarrow\infty$. Hence by Lemma 4.4.1 and Theorem 4.4.1 of \citet{giraitis:koul:surgailis:2009:large}, the tightness in $D[0,1]$ is proved.
\end{proof}

\subsection{Proof of Theorem \ref{Thm:linear}}\label{Sec:proof linear}
\begin{proof}
Set for simplicity $L(n)=c$.
The covariance $\gamma(n)$ for $n>0$ is
\begin{align*}
\gamma(n)=\E X(n)X(0)=\sum_{i=1}^\infty a_{i+n}a_i
=  c^2 \sum_{i=1}^\infty (i+n)^{-1}i^{-1}.
\end{align*}
Note that as $n\rightarrow\infty$,
\[
\sum_{i=2}^\infty (i+n)^{-1}i^{-1}=n^{-1} \sum_{i=2}^\infty \left(\frac{i}{n}+1\right)^{-1} \left(\frac{i}{n}\right)^{-1} \frac{1}{n}= n^{-1}\int_{2/n}^\infty \left(\frac{[nx]}{n}+1\right)^{-1} \left(\frac{[nx]}{n}\right)^{-1 }dx\sim n^{-1}\ln n.
\]
The last asymptotic can be seen from:
\[
\int_{2/n}^\infty \left(x+1\right)^{-1} x^{-1}dx \le \int_{2/n}^\infty \left(\frac{[nx]}{n}+1\right)^{-1} \left(\frac{[nx]}{n}\right)^{-1 }dx \le \int_{1/n}^\infty \left(y+1\right)^{-1} y^{-1}dy,
\]
where we have used the fact $x-1/n\le [nx]/n\le x$, and
 both the lower and upper bounds are asymptotically equivalent to $\ln n$ as $n\rightarrow\infty$.

Hence
\begin{equation}\label{eq:asymp gamma linear}
\gamma(n)\sim c^2 n^{-1}\ln n \quad\text{as $n\rightarrow\infty$.}
\end{equation}
So as $N\rightarrow \infty$, one has
\begin{align}\label{eq:allan var linear}
\E \left(\sum_{n=1}^N X(n)\right)^2&=N \sum_{n=-N+1}^{N-1}\gamma(n)- \sum_{n=-N+1}^{N-1} |n|\gamma(n)\notag\\&\sim 2c^2 N\sum_{n=1}^{N}n^{-1}\ln n\sim 2c^2 N \int_1^N x^{-1}\ln x dx \sim 2c^2 N(\ln N)^2.
\end{align}
Note that by (\ref{eq:asymp gamma linear}) the term $ \sum_{n=-N+1}^{N-1}|n|\gamma(n)= O(N\ln N)$ and is thus negligible. Having obtained the  asymptotic variance (\ref{eq:allan var linear}), the proof is then concluded by applying
\citet{davydov:1970:invariance} Theorem 2 (though this theorem was stated for a linearly interpolated version of $Y_N(t)$ in the space $C[0,1]$, it is straightforward to adapt the the proof, which consists of showing convergence of finite-dimensional distributions and establishing tightness by moment estimate, to establish convergence in $D[0,1]$ with the uniform metric.)
\end{proof}

\begin{Rem}
One may wonder if it is possible to get a different normalization in the nonlinear case in Theorem \ref{Thm:boundary k>=2}, since the normalization  in the linear case in Theorem \ref{Thm:linear} has an extra $\sqrt{\ln N}$ factor. This is not possible under our setting where the kernel $g$ is homogeneous with exponent $\alpha$ and  is bounded by a linear combination of products of \emph{purely} power functions $x_1^{\gamma_1}\ldots x_k^{\gamma_k}$,  where each $\gamma_j<-1/2$ and $\gamma_1+\ldots+\gamma_k=\alpha$.

Indeed, if one wanted to get some extra logarithmic factor in the covariance $\gamma(n)$, one would set for example $g(x_1,\ldots,x_k)=x_1^{\gamma_1}\ldots x_k^{\gamma_k}$ with $\gamma_k=-1$. But this will not achieve the stated goal. Indeed, by Lemma \ref{Lem：sum restriction}, using contradiction,  we have $\alpha=\gamma_1+\ldots+\gamma_k<-k/2-1/2$, which falls into the short memory regime (see Remark \ref{Rem:short memory exponent}) and thus the normalization is $\sqrt{N}$ as in (\ref{eq:clt}).
\end{Rem}

\noindent\textbf{Acknowledgments.} This work was partially supported by the NSF grant  DMS-1309009 at Boston University. We would also like to thank the referee for his comments.

\bibliographystyle{plainnat}

\begin{thebibliography}{14}
\providecommand{\natexlab}[1]{#1}
\providecommand{\url}[1]{\texttt{#1}}
\expandafter\ifx\csname urlstyle\endcsname\relax
  \providecommand{\doi}[1]{doi: #1}\else
  \providecommand{\doi}{doi: \begingroup \urlstyle{rm}\Url}\fi

\bibitem[Bai and Taqqu(2014)]{bai:taqqu:2013:3-generalized}
S.~Bai and M.S. Taqqu.
\newblock Generalized {H}ermite processes, discrete chaos and limit theorems.
\newblock \emph{Stochastic Processes and Their Applications}, 124\penalty0
  (4):\penalty0 1710--1739, 2014.

\bibitem[Bai and Taqqu(2015)]{bai:taqqu:2014:4-convergence}
S.~Bai and M.S. Taqqu.
\newblock Convergence of long-memory discrete k-th order {V}olterra processes.
\newblock \emph{Stochastic Processes and Their Applications}, 125\penalty0
  (5):\penalty0 2026--2053, 2015.

\bibitem[Breton and Nourdin(2008)]{breton:nourdin:2008:error}
J.C. Breton and I.~Nourdin.
\newblock Error bounds on the non-normal approximation of hermite power
  variations of fractional brownian motion.
\newblock \emph{Electronic Communications in Probability}, 13:\penalty0
  482--493, 2008.

\bibitem[Breuer and Major(1983)]{breuer:major:1983:central}
P.~Breuer and P.~Major.
\newblock Central limit theorems for non-linear functionals of {G}aussian
  fields.
\newblock \emph{Journal of Multivariate Analysis}, 13\penalty0 (3):\penalty0
  425--441, 1983.

\bibitem[Davydov(1970)]{davydov:1970:invariance}
Y.A. Davydov.
\newblock The invariance principle for stationary processes.
\newblock \emph{Theory of Probability and Its Applications}, 15\penalty0
  (3):\penalty0 487--498, 1970.

\bibitem[Dobrushin and Major(1979)]{dobrushin:major:1979:non}
R.L. Dobrushin and P.~Major.
\newblock Non-central limit theorems for non-linear functional of {G}aussian
  fields.
\newblock \emph{Probability Theory and Related Fields}, 50\penalty0
  (1):\penalty0 27--52, 1979.

\bibitem[Giraitis et~al.(2012)Giraitis, Koul, and
  Surgailis]{giraitis:koul:surgailis:2009:large}
L.~Giraitis, H.L. Koul, and D.~Surgailis.
\newblock \emph{Large Sample Inference for Long Memory Processes}.
\newblock World Scientific Publishing Company Incorporated, 2012.

\bibitem[Nourdin and Peccati(2012)]{nourdin:peccati:2012:normal}
I.~Nourdin and G.~Peccati.
\newblock \emph{Normal Approximations With Malliavin Calculus: From Stein's
  Method to Universality}.
\newblock Cambridge Tracts in Mathematics. Cambridge University Press, 2012.

\bibitem[Nourdin et~al.(2010)Nourdin, Peccati, and
  Reinert]{nourdin:peccati:2010:invariance}
I~Nourdin, G.~Peccati, and G~Reinert.
\newblock Invariance principles for homogeneous sums: universality of
  {G}aussian {W}iener chaos.
\newblock \emph{The Annals of Probability}, 38\penalty0 (5):\penalty0
  1947--1985, 2010.

\bibitem[Peccati and Taqqu(2011)]{peccati:taqqu:2011:wiener}
G.~Peccati and M.S. Taqqu.
\newblock \emph{Wiener Chaos: Moments, Cumulants and Diagrams: a Survey With
  Computer Implementation}.
\newblock Springer Verlag, 2011.

\bibitem[Peccati and Zheng(2014)]{peccati:zheng:2014:universal}
G.~Peccati and C.~Zheng.
\newblock Universal {G}aussian fluctuations on the discrete {P}oisson chaos.
\newblock \emph{Bernoulli}, 20\penalty0 (2):\penalty0 697--715, 2014.

\bibitem[Rotar(1979)]{rotar:1979:limit}
V.I. Rotar.
\newblock Limit theorems for polylinear forms.
\newblock \emph{Journal of Multivariate analysis}, 9\penalty0 (4):\penalty0
  511--530, 1979.

\bibitem[Taqqu(1979)]{taqqu:1979:convergence}
M.S. Taqqu.
\newblock Convergence of integrated processes of arbitrary {H}ermite rank.
\newblock \emph{Probability Theory and Related Fields}, 50\penalty0
  (1):\penalty0 53--83, 1979.

\bibitem[Terrin and Taqqu(1991)]{terrin:taqqu:1991:power}
N.~Terrin and M.S. Taqqu.
\newblock Power counting theorem in {E}uclidean space.
\newblock In \emph{Random Walks, Brownian Motion, and Interacting Particle
  Systems}, pages 425--440. Springer Verlag, 1991.

\end{thebibliography}

\end{document}